\documentclass{amsart}

\usepackage{stmaryrd}

\usepackage{amsmath, amsthm, amssymb}

\theoremstyle{definition}

\usepackage{scalerel}
\usepackage{mathrsfs}

\usepackage{tikz}
\usepackage{hyperref}
\usepackage{placeins}

\usepackage{proof}

\usepackage{verbatim}
\usepackage{mathtools}
\usepackage{cleveref}

\usepackage{bbm}

\usepackage[bottom]{footmisc}

\usepackage{enumitem}

\usepackage{biblatex} 
\newtheorem{theorem}{Theorem}
\newtheorem*{theorem*}{Theorem}
\newtheorem{maintheorem}{Main Theorem}
\newtheorem*{maintheorem*}{Main Theorem}

\newtheorem*{maintheorems*}{Main Theorems}

\newtheorem{definition}[theorem]{Definition}
\newtheorem{corollary}{Corollary}[theorem]
\newtheorem{prop}[theorem]{Proposition}
\newtheorem{lemma}[theorem]{Lemma}

\newtheorem{remark}[theorem]{Remark}
\newtheorem{claim}{Claim}[theorem]

\newenvironment{claimproof}[1]{\par\noindent\textit{Proof of the Claim.}\space#1}{\hfill $\blacksquare$}

\newcommand\A{\mathbb{A}}
\newcommand\I{\mathscr{I}}
\newcommand\G{\mathcal{G}}
\newcommand\F{\mathcal{F}}

\newcommand\sym{\textup{sym}}
\newcommand\fix{\textup{fix}}
\newcommand\tc{\textup{TC}}
\newcommand\ord{\textup{ORD}}
\newcommand\Aut{\textup{Aut}}

\renewcommand{\ker}{\textup{ker}}

\newcommand\Mtail{M_\textup{Tail}}

\newcommand\Mcol{M_{\textup{Coll}}}

\newcommand\Mrp{M_{\textup{RP-}}}

\newcommand\ZFU{\textup{ZFU}}
\newcommand\ZFCU{\textup{ZFCU}}
\newcommand\ZFUR{\textup{ZFU}_\textup{R}}

\newcommand\AcardSet{\mathbb{A}_\text{Card}\text{Set}}

\newcommand{\barz}{\bar{z}}

\def\<#1>{\left\langle#1\right\rangle}
\def\[#1]{\left\llbracket#1\right\rrbracket}
\renewcommand{\restriction}{\mathord{\upharpoonright}}

\title{\textbf{Reflection Principles in ZFU}}
\author{Elliot Glazer}
\author{Bokai Yao}

\usepackage{graphicx}

\begin{document}
\maketitle

\begin{abstract}
We separate the Collection Principle, the Reflection Principle, and the Partial Reflection Principle in ZF with urelements (ZFU), despite their equivalence under the Axiom of Choice. In particular, Collection and the Partial Reflection Principle are independent of one another, and Collection together with Partial Reflection does not imply the Reflection Principle. We show that Reflection and Collection are equivalent assuming either the Tail axiom or Small Violations of Choice.\end{abstract}

\section{Introduction}
Reflection principles in set theory are intended to capture the indescribability of the universe of sets: intuitively, no single property can fully characterize the universe as a whole. A standard formalization of this idea is given by the Reflection Principle RP, which asserts that any formula is absolute between the universe and arbitrarily large initial segments of it, that is, transitive sets.
\begin{itemize}
\item [] (RP) $\forall x \exists t (x \subseteq t \land t \text{ is a transitive set} \land \forall v \in t (\varphi(v) \leftrightarrow \varphi^t(v)))$.
\end{itemize}
A seemingly weaker formulation is the Partial Reflection Principle RP$^-$, first considered by Lévy \cite{levy1966principles}, which states that whenever a formula (with parameters) holds, it already holds in some transitive set.
\begin{itemize}
\item [] (RP$^-$) $\forall x ((\varphi (x) \rightarrow \exists t (x \in t \ \land \ t \text{ is a transitive set} \land \ \varphi^t(x))).$
\end{itemize}
\noindent A classical result of L\'evy \cite{levy1960axiom} and Montague \cite{Montague1961-MONFAT} shows that RP, and hence RP$^-$, is provable in ZF.

The axioms of ZF, however, rule out \textit{urelements}, that is, elements of sets but are not themselves sets. Once urelements are admitted, the behavior of reflection principles becomes a natural question to explore. This motivates a systematic study of reflection principles in ZF with urelements (ZFU), where a proper class of urelements is allowed. It is folklore that RP$^-$ is no longer provable even in ZFCU (ZFU with the Axiom of Choice). Nevertheless, when choice is available, the failure of reflection can be compensated for by the Collection Principle.
\begin{itemize}
\item [] (Collection) $\forall w, p (\forall x \in w \exists y \varphi(x, y, p)   \rightarrow \exists v \forall x \in w \exists y \in v \varphi(x, y, p))$.
\end{itemize}
Collection is provable in ZF and is often regarded as a natural strengthening of Replacement. Moreover, as an easy consequence of RP (Proposition \ref{prop:RP->powerRP-andCollection}), it may itself be viewed as a form of reflection. The following is proved by the second author in \cite{YaoAxiomandForcing}.
\begin{theorem*}
The following are equivalent over ZFCU.
\begin{enumerate}
\item RP
\item RP$^-$
\item Collection \qed
\end{enumerate}
\end{theorem*}
\noindent This naturally raises the question of whether the equivalence persists once the Axiom of Choice is removed.

In Section \ref{section:ZFUresults}, we identify two axioms that unify Collection and RP. One is the Tail axiom introduced in \cite{YaoAxiomandForcing}, while the other is Small Violations of Choice (SVC), an axiom introduced by Blass \cite{blassInjec1979}, which asserts that the Axiom of Choice holds in a forcing extension.
\begin{maintheorem}
Collection is equivalent to RP over ZFU + (Tail $\lor$ SVC).
\end{maintheorem}
\noindent We also show that RP, RP$^-$ and Collection are equivalent assuming every set of urelements is well-orderable (Corollary \ref{cor:equivalenceunderACA}).

In Section \ref{section:independence}, we demonstrate that the use of the Axiom of Choice is essential for the aforementioned equivalence, thereby answering the main question posed in \cite{YaoAxiomandForcing}.
\begin{maintheorem}
The following implication diagram is complete over $\ZFU$.
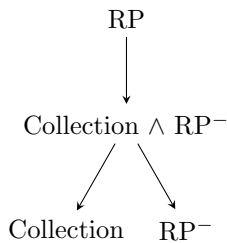
\begin{figure}[hbt!]
\begin{center}
\begin{tikzpicture}[yscale=0.7,xscale=0.8]
\begin{scope}[every node/.style={}]
    \node (C) at (3,3) {RP};

    \node (H) at (3, 1) {\textup{Collection $\land$ RP$^-$}};
    \node (I) at (4, -1) {RP$^-$};
 
    \node (O) at (2, -1) {Collection };
    
\end{scope}

\begin{scope}[>={stealth},
              every node/.style={fill=white,circle},
              every edge/.style={draw=black}]

    \path [->] (C) edge (H);
   \path [->] (H) edge (I);
    \path [->] (H) edge (O);
\end{scope}
\end{tikzpicture}
\end{center}
\end{figure}
\end{maintheorem}
\noindent In particular, RP$^-$ turns out to be a very weak reflection principle in the absence of the Axiom of Choice. This is witnessed by the model $\Mtail$, in which RP$^-$, Tail and SVC all hold while Collection fails. These independence results reveal a structural feature obscured in the pure-set context: reflection principles in the presence of urelements behave as choice principles.

\subsection{Preliminaries} The language of urelement set theory, in addition to $\in$, contains a unary predicate $\A$ for urelements.  The axioms of $\ZFU$\footnote{This theory is also known as ZFA. In the previous work of the second author, this theory is called  $\ZFUR$, where R stands for Replacement.} include Foundation, Pairing, Union, Powerset, Infinity, Separation, Replacement, Extensionality \textit{for sets}, and the axiom that no urelements have members (see \cite[Section 1]{YaoAxiomandForcing} for the precise formulation). Importantly, $\ZFU$ allows a proper class of urelements. We will always work in $\ZFU$ unless stated otherwise.

For every object $x$, $\tc(x)$ denotes the transitive closure of $x$ and $\ker(x)$ (the \textit{kernel} of $x$) is the set of urelements in $\tc(\{x\})$. A set is \textit{pure} if its kernel is empty. $V$ denotes the class of all pure sets. $\ord$ is the class of all ordinals, which are transitive pure sets well-ordered by $\in$. For sets $x$ and $y$, $x \sim y$ ($x$ is equinumerous to $y$) if there is a bijection between $x$ and $y$; $x \preceq y$ if there is an injection from $x$ into $y$; and $x \preceq^* y$ if either there is a surjection from $y$ onto $x$ or $x = \emptyset$. An ordinal is \textit{initial} if it is not equinumerous to any smaller ordinals. For every set $x$, $\aleph(x)$ (\textit{the Hartogs number} of $x$) is the least initial ordinal that does not inject into $x$. $\A$ will also stand for the class of all urelements. $A \subseteq \A$ abbreviates ``$A$ is a set of urelements''; $B \subseteq \A - A$ abbreviates ``$B$ is a set of urelements disjoint from $A$''; and $x \preceq \A - A$ means ``$x$ is equinumerous to some set of urelements disjoint from $A$". For any $A \subseteq \A$, the $V_{\alpha}(A)$- hierarchy is defined in the standard way, i.e., 
\vspace{4pt}
\begin{itemize}
    \item [] $V_0(A) = A$;
    \item [] $V_{\alpha+1}(A) = P(V_{\alpha}(A)) \cup A$;
    \item [] $V_{\gamma}(A) = \bigcup_{\alpha < \gamma} V_\alpha(A)$, where $\gamma$ is a limit;
    \item [] $V(A) = \bigcup_{\alpha \in \ord} V_\alpha(A)$.
\end{itemize}
\vspace{4pt}
For every object $x$ and set $A \subseteq \A$, $x \in V(A)$ if and only if $\ker(x) \subseteq A$. $U = \bigcup_{A\subseteq\A} V(A)$ is the class of all objects. Every permutation $\pi$ of a set of urelements can be extended to an automorphism of $U$ canonically by letting $\pi$ pointwise fix every $a \in \A - dom(\pi)$ and let $\pi x$ be $\{\pi y : y \in x \}$ for every set $x$; for every set $x$, $\pi$ pointwise fixes $x$ whenever $\pi$ pointwise fixes $\ker(x)$. For any inner model $M$, $\Aut(M)$ denotes the class of all automorphisms of $M$. Finally, the following will be useful for verifying Collection in many situations.
\begin{prop}\label{prop:CollectionEquivs}
The following are equivalent over $\ZFU$.
\begin{enumerate}
\item Collection
\item $\forall w, p (\forall x \in w \ \exists y \varphi(x, y, p)   \rightarrow \exists A \subseteq \A \forall x \in w\  \exists y \in V(A) \  \varphi(x, y, p))$
\end{enumerate}
\end{prop}
\begin{proof}
It is clear that (1) $\to$ (2) since every set is in some $V(A)$. Assume (2) and suppose that $\forall x \in w \ \exists y \varphi(x, y, p)$. Fix some $A \subseteq \A$ such that $\forall x \in w\  \exists y \in V(A) \  \varphi(x, y, p)$. For every $x \in w$, let $\alpha_x$ be the least ordinal such that $\exists y \in V_\alpha(A) \varphi(x, y, p)$ and let $\alpha = \sup_{x \in w} \alpha_x$. Then $V_\alpha (A)$ is a collection set.
\end{proof}

\section{ZFU results}\label{section:ZFUresults}
To start with, we introduce certain weak choice principles in ZFU and review some known facts about how they interact with reflection principles, which will be used to prove our first main theorem.

\begin{definition}
\textit{Cardinality is parametrically definable} if there is a parametrically definable class function  $x \mapsto |x|$ such that for every set $x$ and $y$, $x \sim y$ if and only if $|x| = |y|$ ($x \mapsto |x|$ is said to be a \textit{cardinality assignment function}).
\end{definition}
\noindent Gauntt \cite{gauntt1967undefinability} and L\'evy \cite{levy1969definability} independently proved that it is consistent with $\ZFU$ that cardinality is not parametrically definable.
\begin{definition}
A set $A$ of urelements is \textit{universal} if every $B \subseteq \A$ injects into $V(A)$. \textit{Uni} is the axiom that there is a universal set. 
\end{definition}

\begin{lemma}\label{lemma:uni->HP}
Assume Uni.
\begin{enumerate}
\item Cardinality is parametrically definable. 

\item Every set of cardinals has an upper bound.
\end{enumerate}
\end{lemma}
\begin{proof}
(1) Let $A$ be a universal set. For any set $x$, an injection $f$ from $\ker(x)$ into $V(A)$ can be naturally extended to an injection from $\tc(x)$ into $V(A)$ by letting $f(y) = \{ f(z) \mid z \in y\}$ for every set $y \in \tc(x)$. Thus, every set injects into $V(A)$, and so we can use Scott's trick within $V(A)$ by setting $|x|_\text{Scott}$ to be the set of equinumerous sets of $x$ in $V(A)$ with the least rank.

(2) The cardinality assignment function defined using Scott's trick has the property that $\forall x \exists y \in |x|_\text{Scott} (x \sim y)$. So if $X$ is a set of cardinals under this cardinality assignment function, then $|\bigcup \bigcup X|_\text{Scott}$ is an upper bound of $X$. Given any cardinality assignment function $x \mapsto |x|$, the map $|x| \mapsto |x|_\text{Scott}$ is a natural isomorphism. So it follows that every set of cardinals under any cardinality assignment function has an upper bound.\end{proof}

\begin{definition}
Let $A$ be a set of urelements.
\begin{enumerate}
\item  $B$ is a \textit{duplicate} of $A$ if $B \subseteq \A - A$ and $A \sim B$.

\item Dup$(A)$  (\textit{duplication holds over} A) if  for every $B \subseteq \A - A$, there is a $C \subseteq \A - A$ that is a duplicate of $B$.

\item Dup$^+(A)$  (\textit{strong duplication holds over} A) if  for every $B \subseteq \A - A$, there is an infinite family $F$ of pairwise disjoint sets such that for each $C \in F$, $C \subseteq \A - A$ and $C$ has a subset that duplicates $B$.

\item Dup abbreviates $\exists A \subseteq \A \ \text{Dup}(A)$.

\item Dup$^+$ abbreviates $\exists A \subseteq \A \ \text{Dup}^+(A)$. 
\end{enumerate}
\end{definition}
\noindent It is consistent with $\ZFU$ that Dup fails (e.g., Dup fails in $\Mtail$ constructed in the next section). Let AC$^\A$ abbreviate the axiom that every set of urelements is well-orderable.
\begin{lemma}\label{lemma:dup(A)basics}
Let $A$ and $B$ be sets of urelements.
\begin{enumerate}
\item Dup($A$) $\land$ $A \subseteq B$ $\to$ Dup($B$). 
\item Dup$^+(A)$ $\land$ $A \subseteq B$ $\to$ Dup$^+(B)$. 
\item AC$^\A$ $\to$  Dup.
\item AC$^\A$ $\land$ Collection $\to$ Dup$^+$.
\end{enumerate}
\end{lemma}
\begin{proof}
See \cite[Proposition 11 and Lemma 12]{Yao2025Plenitude}.
\end{proof}

\begin{lemma}\label{lemma:Uni+Dup+Col->RP}
Assume that cardinality is parametrically definable. Then Dup$^+$ and Collection jointly imply RP. 
\end{lemma}
\begin{proof}
See \cite[Lemma 16]{Yao2025Plenitude}.
\end{proof}

\subsection{Tail $\land$ Collection $\to$ RP}
\begin{definition}
Let $A$ be a set of urelements.
\begin{enumerate}
\item $A$ is \textit{maximal} if every $B \subseteq \A$ injects into $A$.
\item $T$ is a \textit{tail} of $A$, if $T \subseteq \A - A$  and every $C \subseteq \A - A$ injects into $T$.
\item \textit{Max} is the axiom that there is a maximal set of urelements.
\item \textit{Tail} is the axiom that every set of urelements has a tail.
\end{enumerate}
\end{definition}
\noindent It is shown in \cite{YaoAxiomandForcing} that Tail implies RP over ZFCU. Here we show that Tail and Collection jointly imply RP in the choiceless context (while neither of them implies RP, as shown in the next section). 
\begin{lemma}\label{Lemma:Tail+Col->Dup+}
Tail $\land$ Collection $\to$ Dup$^+$
\end{lemma}
\begin{proof}
Let $E$ be a tail of $\emptyset$ and let $X= \{B \subseteq E \mid \exists C \subseteq \A (B \not\preceq \A - C)\}$. Using Collection on X, we get a set $A \subseteq \A$ such that for each $B \in X$, there is some $C \subseteq A$ with $B \not\preceq \A - C$. We show that Dup$^+(A)$. Since for every $D \subseteq \A - A$, $T \cup D$ is a tail of $A$ whenever $T$ is a tail of $A$, it suffices to show that every tail of $A$ can be duplicated over $A$ infinitely often. 

Let $T$ be a tail of $A$. Suppose $T \not\preceq \A - (A \cup T)$. Since $T \sim B$ for some $B \subseteq E$, it follows that there is some $C \subseteq A$ such that $B \not\preceq \A - C$, which contradicts the assumption that $T \preceq \A - A$. Thus, there is some $T' \subseteq \A - (A \cup T) $ such that $T' \sim T$. So $(T \cup T') \sim T$ and let $f$ be a bijection from $T \cup T'$ to $T'$. Define $g: T \times \omega \to T'$ as $g (d, n) = f^{n+1}(d)$ for every $\<d, n> \in  T \times \omega$. It is routine to check that $g$ is an injection. This shows that $(T' \times \omega) \sim T'$, which produces the desired infinite family of sets.\end{proof}

\noindent Let $\AcardSet$ abbreviate ``Cardinality is parametrically definable and $\{|A| \mid A \subseteq \A\}$ is a set". Note that the conjunction Uni$\,\wedge \, \AcardSet$ is equivalent to the existence of a set which every set of urelements injects into.
\begin{theorem}\label{thm:Tail+Collection->RP}\
\begin{enumerate}
\item Tail $\to$ Max  $\to$ Uni$\,\wedge \, \AcardSet$ $\to \AcardSet$.

\item None of the implications in (1) can be reversed.

\item Assume Collection. All principles in (1) are equivalent and imply RP.

\end{enumerate}

\end{theorem}
\begin{proof}
(1) Tail implies Max because any tail of $\emptyset$ is maximal. Max implies Uni and $\AcardSet$ because if $B \subseteq \A$ is maximal, it is universal so cardinality is parametrically definable by Lemma \ref{lemma:uni->HP}, and $\{|A| \mid A \subseteq \A\} = \{|A| \mid A \subseteq B\}$, which is a set.

(2) The model in \cite[Theorem 2.17 (4)]{YaoAxiomandForcing} shows that Max does not imply Tail over ZFCU, and Uni $\land$ $\AcardSet$ does not imply Max over ZFCU by the model in \cite[Theorem 2.17 (3)]{YaoAxiomandForcing}. Lastly, in the model constructed by the first author in \cite{GlazerMOFUnboundedCards}, $\AcardSet$ holds but the set $\{|A| \mid A \subseteq \A\}$ has no upper bound, which implies that Uni fails in the model by Lemma \ref{lemma:uni->HP}.

(3) Assume Collection and $\AcardSet$. We show that Tail holds. For any $A \subseteq \A$, since $X = \{|B| \mid B \subseteq \A - A\}$ is a set, by Collection there is a set $v$ such that for each $|B| \in X$ there is some $B' \subseteq \A - A$ in $v$ with $|B'| = |B|$. It follows that $T = \bigcup\{B \in t \mid B \subseteq \A - A \}$ is a tail of $A$. Now assume Collection and Tail. Then Uni holds by (1), so it follows from Lemma \ref{lemma:uni->HP}, \ref{lemma:Uni+Dup+Col->RP} and \ref{Lemma:Tail+Col->Dup+} that RP holds.
\end{proof}

\begin{corollary}\label{cor:V(A)modelsRP}
For every $A \subseteq \A$, $V(A) \models $ RP.
\end{corollary}
\begin{proof}
It is clear that $V(A)$ is an inner model of ZFU. $V(A) \models \A \text{ is a set}$ since $A \in V(A)$, so $V(A) \models $ Tail and $V(A) \models$ Collection by Proposition \ref{prop:CollectionEquivs}. Therefore, $V(A) \models$ RP by Theorem \ref{thm:Tail+Collection->RP}.
\end{proof}

\subsection{SVC $\land$ Collection $\to$ RP}
\begin{definition}
Let $S$ be a set. 
\begin{enumerate}
\item SVC(S) if for every set $x$ there is an ordinal $\alpha$ such that $x \preceq^* S \times \alpha$.
\item SVC$^\A$(S) if for every $A \subseteq \A$ there is an ordinal $\alpha$ such that $A \preceq^* S \times \alpha$.
\item SVC abbreviates $\exists S\ $SVC$(S)$, and SVC$^\A$ abbreviates $\exists S\ $SVC$^\A(S)$
\end{enumerate}
\end{definition}
\noindent Blass \cite{blassInjec1979} showed that SVC is equivalent to ``AC holds in a forcing extension''. By a similar argument, one can show that SVC$^\A$ is equivalent to ``AC$^\A$ holds in a forcing extension''. SVC is strictly stronger than SVC$^\A$ as the latter holds whenever $\A$ is a set.
\begin{prop}\label{prop:SVC->Uni}
SVC$^\A$ $\to$ Uni.
\end{prop}
\begin{proof}
Suppose  SVC$^\A$(S) for some set $S$. For every $A \subseteq \A$, since $ A \preceq^* S \times \alpha$ for some ordinal $\alpha$, $A$ injects into $P(S \times \alpha)$ and hence into $V(\ker(S))$. This shows that $\ker(S)$ is universal. \end{proof}

\noindent We review the following fact to verify that both SVC and SVC$^\A$ are equivalent to their injective version.
\begin{prop}
For every set $x$, $S$ and ordinal $\alpha$,  $x \preceq P(S) \times \alpha$ if  $x \preceq^* S \times \alpha$.
\end{prop}
\begin{proof}
Let $f$ be a surjection from $S \times \alpha$ onto $x$. For every $y \in x$, let $\beta_y$ be the least ordinal such that $(S \times \{\beta_y\}) \cap f^{-1}[\{y\}] \neq \emptyset$, and let $S_y = \{ s \in S \mid \<s, \beta_y> \in f^{-1}[\{y\}]\}$. The map $y \mapsto \<S_y, \alpha_y>$ is an injection from $x$ into $P(S) \times \alpha$.
\end{proof}

\begin{lemma}\label{lemma:SVC+Col->Dup+}
SVC$^\A$ $\land$ Collection $\to$ Dup$^+$
\end{lemma}
\begin{proof}
Let $S$ be such that for every $A \subseteq \A$, $A \preceq S \times \alpha$ for some ordinal $\alpha$. A set is said to be \textit{realized} if it is equinumerous to a set of urelements; a subset $P$ of $S$ is \textit{limited} if $P \times \alpha$ is not realized for some ordinal $\alpha$. For every \textit{limited} $P\subseteq S$, let $\lambda_P$ be the least ordinal such that $P \times \lambda$ is not realized. Define
$$ \kappa^* = \sup \{\lambda_P \mid P \subseteq S \land P \text{ is limited} \}.$$
Let $\kappa= \max\{\omega, \kappa^*\}$. By applying Collection to the set $X = \{ Q \subseteq S \times \kappa\ \mid \exists B \subseteq \A (Q \not\preceq \A - B)\}$, we get a set of urelements $A$ such that for every $Q \in X$, $Q \not\preceq \A - A$.

We show that Dup$^+(A)$. Suppose $B \cap A = \emptyset$. Consider the set $Y = \{ Q \subseteq S \times \kappa\ \mid \exists C \subseteq \A - A (C \sim Q)\}$. Using Collection, we can get a set $B^*$  such that (i) $B \subseteq B^* \subseteq \A - A$ and (ii) every $Q \in Y$ injects into $B^*$. Then there is a bijection $F: B^* \to Q$ such that $Q \subseteq S \times \lambda$ for some ordinal $\lambda$. For each $s \in S$, we may identify $Q_s = \{\alpha < \lambda \mid \<s, \alpha> \in Q\}$ as its order type and hence as an ordinal $\alpha_s \leq \lambda$. Let $S_0 = \{s \in S \mid \alpha_s < \kappa\}$, and $S_1 = S - S_0$; let $Q_0 = Q \restriction S_0$, and $Q_1 = Q \restriction S_1$.

We first realize $Q_1$ outside $A \cup B^*$. Since $S_1 \times \kappa \subseteq Q_1$, $S_1$ cannot be a limited subset of $S$ and hence $S_1 \times \alpha$ is realized for every ordinal $\alpha$. Let $\theta$ be an initial ordinal greater than $\max\{\aleph(A \cup B^*), \lambda\}$ and $C \subseteq \A$ be such that $C \sim S_1 \times \theta$.

\begin{claim}
$(C - (A \cup B^*)) \sim C$.
\end{claim}
\begin{claimproof}
Fix some $R \subseteq S_1 \times \theta$ such that $R \sim C \cap (A \cup B^*)$. For every $s \in S_1$, $(\{s\} \times \theta) \cap R$ has size $< \theta$, so $(\{s\} \times \theta) - R$ has size $\theta$ and hence order type $\theta$. Using the unique isomorphism between $(\{s\} \times \theta) - R$ and $\{s\} \times \theta$ for every $s \in S_1$, we can construct a bijection from $(S_1 \times \theta) - R$ to $S_1 \times \theta$, which proves the claim.\end{claimproof}

\noindent Since $Q_1 \subseteq S_1 \times \lambda \subseteq S_1 \times \theta$, we can find some $C_1$ such that $C_1 \cap (A \cup B^*) = \emptyset$ and $C_1 \sim Q_1$. 

Next we realize $Q_0$ outside $A \cup B^* \cup C_1$. $Q_0 \subseteq S \times \kappa$ and $Q_0$ is realized outside $A$ since $Q_0 \preceq B^*$, so $Q_0 \preceq \A - B$ for every $B\subseteq \A$. In particular, $Q_0$ injects into $\A - (A \cup B^* \cup C_1)$. Thus, there is some $C_0$ disjoint from  $A \cup B^* \cup C_1$ such that $C_0 \sim Q_0$.

Now $C_0 \cup C_1$ is a duplicate of $B^*$ that is disjoint from $A$. As in the proof of Lemma \ref{Lemma:Tail+Col->Dup+}, it remains to show that $B^* \times 2 \sim B^*$, which implies $B^* \times \omega \sim B^*$. So it suffices to show that $Q \times 2 = (Q_0 \times 2) \cup (Q_1 \times 2)$ injects into $Q$. For every $s \in S_1$, $\alpha_s$ is an infinite ordinal so there is some injection $G$ that maps $\alpha_s \times 3$ into $\alpha_s$ for each $s \in S_1$. This induces an injection from $Q_1 \times 3$ to $Q_1$ and so $Q_1 \sim Q_1 \times 3$. So $$Q = Q_0 \cup Q_1 \sim Q_0 \cup (Q_1 \times 3) \sim Q_0 \cup Q_1 \cup (Q_1 \times 2) = Q \cup (Q_1 \times 2).$$
It is thus left to show $Q_0 \times 2$ injects into $Q$. For each $s \in S_0$, since $\alpha_s < \kappa$, $\alpha_s + \alpha_s < \kappa$. So the function $g : Q_0 \times 2 \to S \times \kappa$ defined as 
\begin{equation*}
g(s, \beta, i) = 
\begin{cases*}
\<s, \beta>   & $i = 0$ \\
\<s, \alpha_s + \beta> & i =1
\end{cases*}
\end{equation*}
is an injection. As $Q_0 \times 2$ injects into $B^* \cup C_0$, it follows that $g[Q_0 \times 2] \in Y$. Therefore, $g[Q_0 \times 2]$ injects into $B^*$ and hence $Q_0 \times 2$ injects into $Q$. This completes the proof.
\end{proof}

\begin{theorem}\label{thm:SVCA+Collection->RP}
SVC$^\A$ $\land$ Collection $\to$ RP.
\end{theorem}
\begin{proof}
Combine Proposition \ref{prop:SVC->Uni}, Lemma \ref{lemma:uni->HP}, \ref{lemma:Uni+Dup+Col->RP} and \ref{lemma:SVC+Col->Dup+}.
\end{proof}
\noindent Our first main theorem then follows from Theorem \ref{thm:Tail+Collection->RP} and \ref{thm:SVCA+Collection->RP}. As a corollary, we show that the equivalence of the three forms of reflection holds under AC$^\A$.

\begin{lemma}\label{lemma:Uni+Dup+RP-->Col}
Uni $\land$ Dup $\land$ RP$^-$ $\to$ Collection.
\end{lemma}
\begin{proof}
By Lemma \ref{lemma:dup(A)basics}, there is some universal set $B$ over which duplication holds. Assume that $\forall x \in w \exists y \varphi(x, y, p)$ for some $w$ and $p$. Let $A = B \cup \ker(w) \cup \ker(p)$. 

\begin{claim}
If $C \cup D \subseteq \A - A$ and $C \sim D$, then there is an automorphism $\pi$ such that $\pi$ pointwise fixes $A$ and $\pi C = D$.
\end{claim}
\begin{claimproof}
By Lemma \ref{lemma:dup(A)basics} again, we have Dup$(A)$ so we can first fix some duplicate $E$ of $C \cup D$ that is disjoint from $A$. Let $\pi_0$ swap $C$ with some subset of $E$, and let $\pi_1$ swap $\pi_0 C$ with $D$. $\pi_1\pi_0$ is thus a desired automorphism. \end{claimproof}

For every $x \in w$, define
$$\alpha_x = \min \{\alpha \mid \exists z \in V_\alpha (A) \exists y (\varphi(x, y, p) \land (\ker(y) - A) \sim z)\}.$$
Let $\gamma = \sup_{x \in w} \alpha_x$ and $\barz = \{ z \in V_\gamma(A) \mid z \preceq \A - A\}$. By applying RP$^-$ to $\forall z \in \barz \exists C \subseteq \A - A (C \sim z)$, we get a transitive set $t$ such that $\barz \in t$ and $\forall z \in \barz \exists C \in t (C \subseteq \A - A \land C \sim z)$. Let $A^+ = A \cup \bigcup \{C \in t \mid C \subseteq \A - A \}$. 

By Proposition \ref{prop:CollectionEquivs}, it is left to show that $\forall x \in w \exists y \in V(A^+) \varphi(x, y, p)$. Let $x \in w$. Then there is some $y'$ and $z \in \barz$ such that $\varphi(x, y', p)$ and $(\ker(y') - A)  \sim z$. Accordingly, there is some $C \subseteq A^+ - A$ with $C \sim (\ker(y') - A)$. By the claim above, there is an automorphism $\pi$ such that $\pi (\ker(y') - A) = C$ and $\pi$ pointwise fixes $A$. Therefore, $\varphi(x, \pi y', p)$ and $\pi y' \in V(A^+)$. \end{proof}

\begin{corollary}\label{cor:equivalenceunderACA}
The following are equivalent over $\ZFU$ + AC$^\A$.
\begin{enumerate}
\item RP
\item RP$^-$
\item Collection
\end{enumerate}
\end{corollary}
\begin{proof}
Assume AC$^\A$. Then $\emptyset$ is universal and Dup holds by Lemma \ref{lemma:dup(A)basics}, so (2) $\to$ (3) by Lemma \ref{lemma:Uni+Dup+RP-->Col}. SVC$^\A(\{\emptyset\})$ so it follows from Theorem \ref{thm:SVCA+Collection->RP} that  (3) $\to$ (1).
\end{proof}

\section{Independence results}\label{section:independence}
We first consider an intermediate reflection principle, which we call RP$^\sim$. A set is \textit{supertransitive} if it is transitive and contains every subset of its members as a member.
\begin{itemize}
\item [] (RP$^\sim$) $\forall x (\varphi (x) \rightarrow \exists t (x \in t \ \land \ t \text{ is supertransitive} \land \ \varphi^t(x)))$
\end{itemize}

\begin{prop}\label{prop:RP->powerRP-andCollection}
RP $\to$ Collection $\land$ RP$^\sim$.
\end{prop}
\begin{proof}
Assume RP. First note that for any two formulas $\varphi_0$ and $\varphi_1$, we can apply RP to the formula $(v=0 \land \varphi_0) \lor (v=1 \land \varphi_1)$, where $v$ is a new variable, to get a transitive set $t$ such that $\varphi_0^t \leftrightarrow \varphi_0$ and $\varphi_1^t \leftrightarrow \varphi_1$.

For Collection, assume $\forall x \in w \exists y \varphi(x, y, p)$.  Let $t$ be a transitive set extending $\{w, p\}$ which simultaneously reflects the formula $\forall x \in v_0 \exists y \varphi (x, y, v_1)$ and the formula $\varphi(v_0, v_1, v_2)$. It follows that $t$ is a collection set. For RP$^\sim$, suppose that $\varphi (x)$ holds. Let $t$ be a transitive set containing $x$ which simultaneously reflects the formula $\varphi (v_0) \land \text{Powerset}$ and the formula $v_1 = P(v_2)$. Then $t$ is a supertransitive set such that $\varphi^t(x)$ holds.\end{proof}

\noindent We proceed to prove the following independence results, thereby proving the second main theorem and answering Question 2.3 in \cite{YaoAxiomandForcing} negatively.

\begin{enumerate}

\item ZFU + Collection $\nvdash$ RP$^-$

\item ZFU + Collection $+$ RP$^-$ $\nvdash$ RP$^\sim$

\item ZFU + RP$^\sim$ + SVC + Tail $\nvdash$  Collection

\end{enumerate}

\subsection{Permutation models with small kernels}
We will utilize the technique of permutation models, which is the standard way of producing non-well-orderable sets of urelements. To create a proper class of sets as such, we go to an inner model which we call ``a small-kernel model''. This method has appeared in many places, including \cite{levy1969definability}, \cite{blassInjec1979}, \cite{FinitenessSVC} and the previous work of the second author.  However, the first author independently observed that these consistency results can be obtained via relative constructibility inside a model of ZFU in which every set is equinumerous to a set of urelements. 
\begin{definition}\label{permutationmodeldef}
Let $\G$ be a group of permutations of some set of urelements $A$.
\begin{enumerate}
\item For every $x$, $\sym_\G (x) = \{\pi \in \G \mid \pi x = x\}$; $\fix_G (x) = \{\pi \in \G \mid \pi y = y \text{ for all } y\in x \}$.

\item A \textit{normal filter} $\F$ on $\G$ is a nonempty set of subgroups of $\G$ which contains $\sym_\G (a)$ for every urelement $a$ and is closed under supergroup, finite intersection, and conjugation (i.e., for all $\pi \in \G$ and $H \in \F$, $\pi H \pi^{-1} \in \F$).

\item  An object $x$ is $\F$-\textit{symmetric} if $\sym_\G (x) \in \F$ and hereditarily $\F$-symmetric if every $y \in \tc(\{x\})$ is $\F$-symmetric. The \textit{permutation model} generated by $\F$ is the class $W = \{x \mid x \text{ is hereditarily } \F\text{-symmetric}\}.$

\item  An ideal $\I$ on $\A$ is $\G$-\textit{flexible} if (i) $\I$ is closed under subset and finite union, (ii) $\sym_\G(\I) = \G$ and (iii) whenever $A \cup \{a\} \in \I$ and $a \notin A$, there is a $\pi \in \fix_\G(A)$ with $\pi a \neq a$.

\item The \textit{small-kernel permutation model} generated by $\F$ and $\I$ is the class $$M = \{ x \mid x \text{ is hereditarily } \F\text{-symmetric } \land \ker(x) \in \I\}.$$

\end{enumerate}
\end{definition}
\noindent It is a standard result that $\G \subseteq \Aut (W)$ and $W \models \ZFU$. And it follows from the definition that $\G \subseteq \Aut(M)$.

\begin{theorem}\label{thm:FThmofSKPM}
Let $\G, \F$ and $\I$ be as above, and $M$ be the resulted small-kernel permutation model.
\begin{enumerate}
\item $M \models \ZFU$.
\item Assume AC in $U$. $M \models $ SVC if $M \models \text{ There is a universal set}$.

\end{enumerate}
 \end{theorem}
\begin{proof}
(1) $M$ is transitive and contains all pure sets, so Extensionality, Foundation, and Infinity all hold in $M$. It is also routine to verify Pairing, Union, and Powerset all hold in $M$ by the standard argument for permutation models, since these axioms only generate sets whose kernel is small in the sense of $\I$ from sets in $M$.

It remains to show that $M$ satisfies the following form of Replacement, which implies Separation.
$$\forall x, p [\forall x \in w \exists ! y\varphi (x, y, p) \to \exists v \forall y (y \in v \leftrightarrow \exists x \in w \varphi(x, y, p))].$$
Suppose $w, p \in M$ and $M \models \forall x \in w \exists ! y \varphi (x, y, p)$. Let $v = \{ y\in M \mid \exists x \in w \  \varphi(x, y, p)\}.$ Since every $\pi \in G$ is an automorphism of $M$, it follows that $\sym(w) \cup \sym(p) \subseteq \sym(v)$ (we shall omit the subscript $\G$ for this theorem). It remains to show that $\ker(v) \subseteq \ker(w) \cup \ker(p)$. Suppose $a \in \ker(v) -(\ker(w) \cup \ker(p)).$ So $a \in \ker(y)$ and $M \models \varphi (x, y, p)$ for some $y \in v$ and $x \in w$. Since $\I$ is $\G$-flexible, it follows that there is some $\pi \in \fix(\ker(w) \cup \ker(p))$ such that $\pi a \notin \ker(y)$. Then $M \models \varphi(x, \pi y, p)$ and $\pi y \neq y$, which contradicts the assumption.

(2) Let $A$ be a universal set in $M$. Then every set in $M$ injects into $V^M(A)$ by the proof of Lemma \ref{lemma:uni->HP}. So to show SVC holds in $M$, it suffices to find an $S \in V^M(A)$ such that for every $x \in V^M(A)$ there is an ordinal $\alpha$ and a surjection $f \in M$ from $S \times \alpha$ onto $x$. In $U$, let
$\F' = \{H \in \F \mid \exists x \in V^M(A) \ \sym(x) = H \}.$ By AC in $U$, for each $H \in \F'$ we can choose an $x_H \in V^M(A)$ with $\sym(x_H) = H$. Define $$S = \{\pi x_H \mid \pi \in \G ,  H \in \F' , \text{ and } \ker(\pi x_H) \subseteq A\}.$$ If $\sigma \in \sym (A)$ and $\pi x_H \in S$, then $\ker(\sigma \pi x_H) \subseteq \sigma A = A$ so $\sigma \pi x_H \in S$. Thus $\sym(A) \subseteq \sym(S)$, and it follows that $S \in V^M(A)$.

Let $x \in V^M(A)$. In $U$, enumerate $x$ as $\{y_\eta \mid \eta < \alpha\}$ by some ordinal $\alpha$ and define
$$f = \{\<\pi x_{\sym(y_\eta)}, \eta, \pi y_\eta> \mid \pi \in \G, \eta < \alpha, \text{ and } \ker(\pi x_{\sym(y_\eta)} \cup \pi y_\eta) \subseteq A\}.$$
\noindent It follows that $\sym(A) \subseteq \sym(f)$, $dom(f) \subseteq S \times \alpha$ and $\ker(f) \subseteq A$; so $f \in M$. $f$ is a function: suppose $\pi x_{\sym(y_\eta)} = \pi' x_{\sym(y_\eta)}$; then $\pi^{-1} \pi' \in \sym(y_\eta)$ and so $\pi y_\eta = \pi' y_\eta$. $f$ is onto $x$: if $y_\eta \in x$, then $x_{\sym(y_\eta)} \in S$ and so $f(x_{\sym(y_\eta)}, \eta) = y_\eta$.
\end{proof}

\noindent The following definition will be useful for verifying Collection.
\begin{definition}\label{def:TailProperty}
Let $M$ be an inner model of $\ZFU$. $M$ has \textit{the tail property} if for every $A \subseteq \A$ in $M$ there is some $T\subseteq \A - A$ in $M$ such that for every $B \subseteq \A - A$ in $M$, there is some $\pi \in \Aut(M)$ such that $\pi$ pointwise fixes $A$ and $\pi B \subseteq T$.
\end{definition}
\begin{lemma}\label{Lemma:TailProperty->Col}
Let $M$ be an inner model of $\ZFU$ with the tail property. Then $M \models$ Collection.
\end{lemma}
\begin{proof}
Suppose that $M \models \forall x \in w \exists y \varphi(x, y, p)$ for some $w, p \in M$. Let $A = \ker(w) \cup \ker(p)$ and $T \subseteq \A - A$ witness the tail property for $A$. Let $x \in w$ and $y$ be such that $\varphi^M(x, y, p)$. Then there is some $\pi \in \Aut(M)$ such that $\pi \ker(y) \subseteq A \cup T$ and $\pi$ pointwise fixes $A$. So $M \models \varphi(x, \pi y, p)  \land \pi y \in V(A \cup T)$. Thus, $M \models \forall x \in w \exists y \in V(A \cup T) \varphi(x, y, p)$, which suffices for Collection by Proposition \ref{prop:CollectionEquivs}.\end{proof}

\subsection{Collection $\not\to$ RP$^-$}

We work in some $U \models \ZFCU$ in which $\A = \{ a_\alpha \mid \alpha < \omega^2 \}$. For every $\alpha < \omega^2$, let $A_\alpha = \{a_\beta \mid \beta < \alpha\}$.
Define
\begin{enumerate}
\item $\I = \{ E \subseteq \A \mid \exists \alpha < \omega^2 (E \subseteq A_\alpha)\}.$

\item A permutation $\pi$ of $\A$ is \textit{permissible} if $\pi \I = \I$.

\item $\G = \{\pi \mid \pi \text{ and } \pi^{-1} \text{ are permissible permutations of } \A \}$.

\item $\F = \{ H \subseteq \G \mid \exists E \in \I \exists S \in  [P(E)]^{<\omega} (\fix_\G (S) \subseteq H) \}.$

\item $\Mcol = \{x \mid x \text{ is $\F$-hereditarily symmetric} \land \ker (x) \in \I \}.$

\end{enumerate}

\begin{prop}
 $\F$ is a normal filter on $\G$ and $\I$ is $\G$-flexible. \qed
\end{prop}
\noindent It follows from Theorem \ref{thm:FThmofSKPM} that $\Mcol\models \ZFU$.

\begin{lemma}
$\Mcol \models$ Collection.
\end{lemma}
\begin{proof}
By Lemma \ref{Lemma:TailProperty->Col}, it suffices to show that $\Mcol$ has the tail property (Definition \ref{def:TailProperty}). Let $\alpha < \omega^2$ and $T = A_{\alpha + \omega} - A_\alpha$. If $B \in \Mcol$ is disjoint from $A_\alpha$ then there is some countable $E \in \I$ that is disjoint from $T \cup B$. Let $\pi_0$ swap $B$ and a subset of $E$ and $\pi_1$ swap this subset with a subset of $T$. Both $\pi_0$ and $\pi_1$ are in $\G$ so $\pi_1 \pi_0$ is a desired permutation.\end{proof}

If $x \in \Mcol$ and $\fix_\G (S) \subseteq \sym_\G (x)$, then $S$ is said to be a \textit{support} of $x$. For every finite family $S$ of sets, we say $R$ is a \textit{refinement} of $S$ if $R$ is a finite partition of $\bigcup S$ such that for every $p \in R$ and $s \in S$, if $s \cap p \neq \emptyset$, then $p \subseteq s$. 
\begin{prop}\label{prop:FinRefine}
Every finite family of sets has a refinement.
\end{prop}
\begin{proof}
By induction on $\omega$. Given $\{s_i \mid i \leq n\}$, suppose that $\{s_i \mid i < n\}$ has a refinement $\{q_j \mid j < k \}$. Then
$$R = \{s_n - \bigcup_{i<n} s_i\} \cup \{s_n \cap q_j \mid j < k\} \cup \{q_j - s_n \mid j < k\}$$
is a refinement of $\{s_i \mid i \leq n\}$
\end{proof}

\begin{lemma}
For every $x \in \Mcol$, there is some $\alpha < \omega^2$ and a finite partition $P$ of $A_\alpha$ such that $\ker(x) \subseteq A_\alpha$ and $P$ supports $x$.
\end{lemma}
\begin{proof}
Let $S$ be a support of $x$ such that $S \in [P(A_\beta)]^{<\omega}$ for some $\beta < \omega^2$, and let $\gamma$ be such that $\ker(x) \subseteq A_\gamma$. Set $\alpha = \max \{\beta, \gamma\}$. Let $R$ be a refinement of $S$ and $P = R \cup \{A_\alpha - \bigcup R\}$, which is a finite partition of $A_\alpha$. For any $\pi \in \fix_\G(P)$ and $a \in E \in S$, $a \in B$ for some $B \in R$; so $B \subseteq E$ and hence $\pi a \in E$. This shows that $\pi E = E$. Therefore, $\fix_\G (P) \subseteq \fix_\G (S)$, making $P$ a support of $x$.
\end{proof}

\begin{lemma}
$\Mcol \not\models$ RP$^-$.
\end{lemma}
\begin{proof}
Let $\varphi$ be the sentence
$$\forall B \exists A \subseteq \A - B (A \text{ is infinite}),$$
which holds in $\Mcol$ because $A_{\alpha + \omega} - A_\alpha$ is in $\Mcol$ for every $\alpha < \omega^2$. Suppose \textit{for reductio} that $t \models \varphi$ for some transitive $t \in \Mcol$ which satisfies enough amount of ZFU. Fix some $A_\alpha$ and a finite partition $P$ of $A_\alpha$ such that $\ker(t) \subseteq A_\alpha$ and $P$ supports $t$. Let $m$ be the size of the set $P - t$. Then in $t$, there is a family $X$ of $m+1$-many pairwise disjoint infinite sets of urelements such that $\bigcup X \cap \bigcup (P\cap t) = \emptyset$.

\begin{claim}
For each $A \in X$ and $E \in P - t$, $A \cap E$ is not both infinite and co-infinite in $E$.
\end{claim}
\begin{claimproof}
Suppose otherwise. Let $\pi$ only swap $A \cap E$ and $E - A$. It follows that $\pi \in \fix_\G(P)$. Thus, $\pi t = t$ and $\pi A \in t$. But $E = A \triangle \pi A$ and $t$ is closed under symmetric difference, so $E \in t$---contradiction.\end{claimproof}
\vspace{6pt}

\noindent It follows that for some $E \in P - t$, there are two disjoint $A, A' \in X$ such that $A \cap E$ and $A' \cap E$ are both co-finite, which is a contradiction.\end{proof}

\subsection{Collection $\land$ RP$^-$ $\not\to$ RP$^\sim$}
We start in some model $U \models \ZFCU$ in which $$\A = \bigsqcup_{\alpha <\omega^2} A_\alpha$$ where each $A_\alpha = \{a_{\alpha, n} \mid n < \omega\}$. That is, we replace each urelement in $\Mcol$ with a countable set of urelements. So $\A$ is viewed as a matrix of which each $A_\alpha$ is a row. For each $\alpha < \omega^2$, let $D_\alpha = \bigcup_{\beta < \alpha}A_\beta$.

\begin{enumerate}
\item $E \subseteq \A$ is \textit{short} if $\{\alpha < \omega^2 \mid E \cap A_\alpha \neq \emptyset\}$ is finite.

\item  $B \subseteq \A$ is \textit{narrow} if for every $\alpha < \omega^2$, $B \cap A_\alpha$ is finite. 

\item $J = \{ X  \mid  X \in [P(D_\alpha)]^{<\omega} \text{for some } \alpha < \omega^2  \text{ and } \forall B \in X (B \text{ is narrow}) \}.$

\item $I =\{ E \cup X \mid E \text{ is short }\land X \in J \}$.

\item $\I = \{ C \subseteq \A \mid \exists \alpha < \omega^2 (C \subseteq D_\alpha)\}$.

\item A permutation $\pi$ of $\A$ is \textit{permissible} if 

\subitem (i) $\pi A_\alpha$ is short for every $\alpha < \omega^2$; and

\subitem (ii) $\pi \I = \I$.

\item $\G = \{\pi \mid \pi \text{ and } \pi^{-1} \text{ are permissible permutations of } \A \}$.

\item $\F = \{H \subseteq \G \mid \exists S \in I (\fix_\G (S) \subseteq H)\}$

\item $\Mrp = \{ x \mid x \text{ is hereditarily $\F$-symmetric } \land \ker(x) \in \I\}$.
\end{enumerate}
\begin{prop}
$\F$ is a normal filter on $\G$ and $\I$ is $\G$-flexible. 
\end{prop}
\begin{proof}
For every $\pi \in \G$, short set $E$ and narrow set $B$, $\pi E$ is short and $\pi B$ is narrow, and so $\pi I = I$. It follows that $\F$ is a normal filter. Any permutation of $\A$ that swaps only two urelements is in $\G$, so $\I$ is $\G$-flexible.
\end{proof}

\noindent So $\Mrp \models \ZFU$ by Theorem \ref{thm:FThmofSKPM}.  Every $A_\alpha$ is countable in $\Mrp$ as $\fix_\G (A_\alpha) \in \F$. Since every $X \in J$ has a refinement (Proposition \ref{prop:FinRefine}), it follows that every $x \in \Mrp$ has a support of the form $E \cup P$ (i.e., $\fix_\G (E \cup P) \subseteq \sym(x)$), where $E$ is a short set and $P$ is a finite family of pairwise disjoint narrow subsets of some $D_\alpha$.

\begin{lemma}\label{lemma:NarrowSwap}
Let $\alpha < \omega^2$ and $B_0$ and $B_1$ be two equinumerous narrow sets in $\Mrp$. If $B_0 \cup B_1 \subseteq \A - D_\alpha$, then there is some $\pi \in \fix_\G (D_\alpha)$ with $\pi B_0 = B_1$.
\end{lemma}
\begin{proof}
First, every permutation that only swaps two disjoint narrow sets in $\Mrp$ is in $\G$ because it only moves finitely many urelements on each $A_\alpha$. Moreover, we can always find a narrow set $C$ disjoint from $B_0 \cup B_1 \cup D_\alpha$. So let $\pi_0$ swap $B_0$ and $C$, and $\pi_1$ swap $C$ and $B_1$. $\pi_1 \pi_0$ is as desired. \end{proof}

\begin{definition}\label{def:Talpha}
For every $\alpha < \omega^2$, $B_{\alpha,0} = \{a_{\beta, 0} \mid \alpha < \beta < \alpha + \omega\}$ and $T_\alpha = A_\alpha \cup B_{\alpha,0}$.
\end{definition}

\begin{lemma}\label{lemma:TalphaTail}
For every $\alpha < \omega^2$ and $D \subseteq \A - D_\alpha$ in $\Mrp$, there is some $\pi \in \fix_\G(D_\alpha)$ such that $\pi D \subseteq T_\alpha$. Hence, $\Mrp \models$ Collection.
\end{lemma}
\begin{proof}
We first show that $\{\alpha < \omega^2 \mid D \cap A_\alpha \text{ is infinite}\}$ is finite. Otherwise, let $E \cup P$ be a support of $D$. There is some $\alpha$ such that $D\cap A_\alpha$ is infinite and $E \cap A_\alpha = \emptyset$. $D \cap A_\alpha - \bigcup P$ is infinite, so let $\pi$ swap some $a \in D \cap A_\alpha - \bigcup P$ with some $b \notin D \cup E \cup \bigcup P$. Then $\pi \in \fix_\G(E \cup P)$ and so $\pi b \in D$---contradiction.

Now let $\gamma < \omega^2$ be such that $D \subseteq D_\gamma$ and $X = \{\beta < \gamma \mid A_\beta \cap D \text{ is infinite}\}$. We may assume $\alpha \in X$. Let $\sigma_0$ swap $\bigcup_{\beta \in X} A_\beta$ with $A_\gamma$ and $\sigma_1$ swap $A_\gamma$ with $A_\alpha$. Since $X$ is finite, $\sigma_1 \sigma_0 \in G$. Furthermore, since the set $D' = D - \bigcup_{\beta \in X} A_\beta$ is narrow, by Lemma \ref{lemma:NarrowSwap} there is some $\rho \in \fix_\G (D_{\alpha + 1})$ and $\rho D' \subseteq B_{\alpha,0}$. It is easy to check that $\pi = \rho \sigma_1\sigma_0$ is a desired permutation. It follows that $\Mrp$ has the tail property: for every $A \in \Mrp$ such that $A \subseteq D_\alpha$ for some $\alpha < \omega^2$, $T_\alpha$ is a desired set. Hence, $\Mrp \models$ Collection by Lemma \ref{Lemma:TailProperty->Col}. \end{proof}

\begin{lemma}\label{lemma:narrow<->D-fin}
For every set of urelements $B$ in $\Mrp$, $B$ is infinite and narrow if and only if $\Mrp \models \aleph (B) = \omega$.
\end{lemma}
\begin{proof}
Assume $B$ is infinite and narrow. Suppose that $f \in \Mrp$ is an injection from $\omega$ into $B$, which has a support of the form $E \cup P$. Let $B' = \bigcup \{ f[\omega] \cap A_\alpha \mid E \cap A_\alpha = \emptyset\}$, which is infinite since $E$ is short. Then either $B' - \bigcup P$ has at least two elements, or $B' \cap C$ has at least two elements for some $C \in P$. In either case, we can swap two elements in $f[\omega]$ while fixing $f$---contradiction. On the other hand, if $B \cap A_\alpha$ is infinite, this intersection is countable in $\Mrp$ so $\Mrp \models \aleph(B) \neq \omega$.\end{proof}

\begin{lemma}
$\Mrp \not\models $ RP$^\sim$.
\end{lemma}
\begin{proof}
Let $\varphi$ be the sentence
$$\forall B \subseteq \A \exists A \subseteq \A - B ( \aleph (A) =\omega),$$
which holds in $\Mrp$ because for every $\alpha < \omega^2$, $\Mrp \models \aleph (B_{\alpha, 0}) = \omega$ by Lemma \ref{lemma:narrow<->D-fin}. Suppose \textit{for reductio} that $t \models \varphi$ for some supertransitive $t \in \Mrp$ which satisfies enough amount of ZFU. Let $E \cup P'$ be a support of $t$, and let $P = \{ B' \cap t \mid B' \in P'\}$. Then there is some set of urelements $A \in t$ such that $A \cap \bigcup (P \cap t) = \emptyset$ and $t \models \aleph(A) = \omega$. Since $t$ is supertransitive, $\Mrp \models \aleph(A) = \omega$ and hence $A$ is an infinite narrow set.

\begin{claim}
There is some $B \in P - t$ such that $A \cap B$ is infinite.
\end{claim}
\begin{claimproof}
Otherwise, $A - \bigcup P$ is infinite and narrow so there is some $a \in A - (\bigcup P \cup E)$. Since $\bigcup P' \subseteq D_\alpha$ for some $\alpha < \omega^2$, we can find some $b \notin t \cup  E \cup \bigcup P'$. Swapping $a$ and $b$ then yields a contradiction since this permutation is in $\fix_\G(E \cup P')$. \end{claimproof}

Fix such $B$.  $B - A$ must be infinite and narrow as $B \cap A \in t$. Furthermore, $E$ is short so $B \cap E$ is finite and hence in $t$. So let $\pi \in \G$ swap $B - (A \cup E)$ and $(B \cap A) - E$. Since $\pi \in \fix_\G (E \cup P')$, it follows that $B - (A \cup E) \in t$, which implies $B \in t$---contradiction.\end{proof}

It remains to show that RP$^-$ holds in $\Mrp$. Let $A = T_0 \cup A_1$. We first work inside $\Mrp$ and define a permutation model within $V(A)$. Then we show that there is an isomorphism from $\Mrp$ to this permutation model that pointwise fixes $V(T_0)$. To begin with, since $A_1$ is countable in $\Mrp$, we can partition $A_1$ as $\bigsqcup_{0<\alpha<\omega^2} A^*_\alpha$, where each $A^*_\alpha = \{a^*_{\alpha, n} \mid n < \omega\}$. For each $\alpha < \omega^2$, let $D^*_\alpha = \bigcup_{\beta < \alpha}A^*_\beta$. Now we make the following definitions in $V(A)$.

\begin{enumerate}
\item $E \subseteq \A$ is \textit{short}$^*$ if $\{\alpha < \omega^2 \mid E \cap A^*_\alpha \neq \emptyset\}$ is finite and $E \cap B_0$ is finite.

\item  $B \subseteq \A$ is \textit{narrow}$^*$ if for every $\alpha < \omega^2$, $B \cap A^*_\alpha$ is finite. 

\item $J^* = \{ X  \mid  X\in [P(D_\alpha \cup B_0)]^{<\omega} \text{ for some } \alpha < \omega^2 \text{ and } \forall B \in X (B \text{ is narrow}^*) \}.$

\item $I^* =\{ E \cup X \mid E \text{ is short}^* \land X \in J^* \}$.

\item $\I^* = \{ C \subseteq \A \mid \exists \alpha < \omega^2 (C \subseteq D_\alpha \cup B_0)\}$.

\item A permutation $\pi$ of $A$ is \textit{permissible}$^*$ if 

\subitem (i) $\pi A^*_\alpha$ is short for every $\alpha < \omega^2$; and

\subitem (ii) $\pi \I^* = \I^*$.

\item $\G^* = \{\pi \mid \pi \text{ and } \pi^{-1} \text{ are permissible}^* \text{ permutations of } A \}$.

\item $\F^* = \{H \subseteq \G^* \mid \exists S \in I^* (\fix_{\G^*}(S) \subseteq H)\}$

\item $M^* = \{ x \mid x \text{ is hereditarily $\F^*$-symmetric } \land \ker(x) \in \I^*\}$.\\
\end{enumerate}
In $U$, define a function $\sigma'$ on $\A$ as follows
\begin{equation*}
\sigma' a_{\alpha,n} = 
\begin{cases*}
a_{\alpha, n}   & $a_{\alpha,n} \in T_0$ \\
a^*_{\alpha, n-1}    & $ 0 <\alpha < \omega$ and $n > 0$\\
a^*_{\alpha, n} & $\alpha \geq \omega$
\end{cases*}
\end{equation*}

\noindent The canonical extension of $\sigma'$ is an isomorphism from $U$ to $V(A)$ (in the sense of $U$) that pointwise fixes $T_0$.  Let $\sigma = \sigma' \restriction \Mrp$.

\begin{lemma}
$\sigma$ is an isomorphism from $\Mrp$ to $M^*$. 
\end{lemma}
\begin{proof}
 The key point, which is routine to verify, is that $\sigma$ maps every short set to a short$^*$ set and every narrow set to a narrow$^*$ set, and similarly for $\sigma^{-1}$. So it follows that $\sigma I = I^*$. Next, observe that $\sigma \G = \{\sigma \pi \sigma^{-1} \mid \pi \in \G\}$. So it follows that $\sigma \G \subseteq \G^*$; and if $\rho \in \G^*$, $\pi = \sigma^{-1} \rho \sigma \in \G$ and hence $\rho = \sigma \pi \sigma^{-1}$ is in $\sigma \G$. Thus, $\sigma \G = \G^*$. It is also clear that $\sigma \I = \I^*$, and from these facts it follows that $\sigma \F = \F^*$.

Now we show that $\sigma: \Mrp \to M^*$ by induction. Suppose that $x \in \Mrp$ and $\sigma x \subseteq M^*$. $\sym_\G(x) \in \F$ and $\ker(x) \in \I$, so $\sym_{\G^*}(\sigma x) = \sym_{\sigma\G}(\sigma x) \in \sigma \F = \F^*$ and $\ker(\sigma x) \in \I^*$. Thus, $\sigma x \in M^*$. By a similar induction, $\sigma^{-1} x \in \Mrp$ for every $x \in M^*$. Therefore, $\sigma$ is onto and hence an isomorphism. \end{proof}

\begin{lemma}\label{lemma:partialreflection}
$\Mrp \models$ RP$^-$.
\end{lemma}
\begin{proof}
Suppose that $\Mrp \models \varphi(x)$ from some $x \in \Mrp$. By Lemma \ref{lemma:TalphaTail}, there is a $\pi \in \G$ such that $ ker( \pi x) \subseteq T_0 \subseteq A$. So $\Mrp \models \varphi(\pi x)$. Since $\sigma$ defined above is an isomorphism from $\Mrp$ to $M^*$ that pointwise fixes $T_0$, we have $M^* \models \varphi(\pi x)$. In $\Mrp$, $M^*$ is a definable transitive class in $V(A)$, so by RP$^-$ in $V(A)$ (Corollary \ref{cor:V(A)modelsRP}) there is a transitive set $t^* \in V(A)$ such that $t^* \cap M^* \models \varphi (\pi x)$, where $t^* \cap M^*$ is transitive and contains $\pi x$. Let $t = \pi^{-1} (t^* \cap M^*)$. Then $t$ is a transitive set in $\Mrp$ containing $x$ such that $t \models \varphi (x)$.\end{proof}

\begin{remark}
A construction similar to the proof of Lemma \ref{lemma:partialreflection} shows that $\Mcol$ satisfies a weaker reflection principle, namely that $\phi(x)$ implies $\phi^N(x)$ for some (not necessarily transitive) model $(N, \in, \A')$ where $\tc(\{x\}) \subset N.$ The main distinction here is that urelements of $N$ need not be urelements in $\Mcol.$
\end{remark}

\subsection{RP$^\sim$ $\land$ Tail $\land$ SVC $\not\to$ Collection}
We work in some $U \models \ZFCU$ in which $$\A = \bigsqcup_{n, m, k <\omega} A_{n, m, k},$$ where each $A_{n, m, k}$ has size $\omega_m$ and is enumerated as $\{a_{n, m, k, \alpha} \mid \alpha < \omega_m\}$. Let $A_{n, m} = \bigcup_{k< \omega} A_{n, m, k}$ for each $n, m < \omega$. That is, we partition $\A$ into $(\omega \times \omega)$-many matrices of the form $A_{n, m}$, each of which has $\omega$-many rows of size $\omega_m$.
\begin{enumerate}
\item $E \subseteq \A$ is \textit{locally small} if $|E \cap A_{n, m}| < \omega_m$ for each $n, m <\omega$.

\item A permutation $\pi$ of $\A$ is \textit{row-preserving} if $\pi A_{n, m, k} = A_{n, m, k}$ for every $n, m, k < \omega$; $\pi$ is \textit{uniform} if for every $i, j, k, m, n < \omega$ and $\alpha, \beta < \omega_m$, if $\pi a_{n, m, k ,\alpha} = a_{n, m, k ,\beta}$, $\pi a_{i, m, j ,\alpha} = a_{i, m, j ,\beta}$.

\item $\G = \{ \pi \mid \pi \text{ is a uniform row-preserving permutation of } \A\}$.

\item $\F = \{ H \subseteq \G \mid \exists E \subseteq \A (E \text{ is locally small } \land \fix_\G(E) \subseteq H)\}$

\item For each $n < \omega$, $T_n = \bigcup_{n \leq m} A_{n, m}$.

\item $\I = \{B \subseteq \A \mid \exists n < \omega (B \subseteq \bigcup_{i\leq n} T_i)\}.$

\item $\Mtail = \{x \mid x \text{ is hereditarily $\F$-symmetric} \land \ker(x) \in \I\}.$

\end{enumerate}

\begin{lemma}
$\Mtail \models \ZFU$.
\end{lemma}
\begin{proof}
$\I$ is not $\G$-flexible so Theorem \ref{thm:FThmofSKPM} does not directly apply. But as $\F$ is $\G$-normal, we can first consider the resulted permutation model $W$. Since $\sym_\G(A_{n, m, k}) = \G$ for each $n, m, k < \omega$, the whole structure of $\A$ is preserved in $W$ and hence $\I$ is in $W$. In $W$, let $\G'$ be the group of permutations such that $\pi \in \G'$ if and only if $\pi \I = \I$ and $\pi^{-1} \I = \I$; and let $\F'$ be the trivial normal filter that contains all subgroups of $\G'$. $\I$ is $\G'$-flexible because every permutation that swaps two urelements is in $\G'$. In $W$, define $M = \{ x \mid x \text{ is hereditarily } \F'\text{-symmetric} \land \ker(x) \in \I\}$,  which satisfies $\ZFU$ by Theorem \ref{thm:FThmofSKPM}, and $M = \Mtail$.\end{proof}
By standard arguments of permutation models, each $A_{n, m, k}$ is non-well-orderable in $\Mtail$ and $\Mtail \models \aleph(A_{n, m, k}) = \omega_m$.
\begin{lemma}
$\Mtail \not\models$ Collection.
\end{lemma}
\begin{proof}
Let $\varphi(n, A)$ abbreviate $$A \subseteq \A \land \forall B \subseteq \A - A ( B \text{ is non-well-orderable} \to \aleph(B) > \omega_n).$$ We show that
$$\Mtail \models \forall n < \omega \exists A \ \varphi(n, A) \land  \neg\exists v \forall n < \omega \exists A \in v \ \varphi (n, A),$$
which fails Collection.
For each $n < \omega$, let $A$ be $\bigcup_{i\leq n} T_i$. Suppose $B$ is a non-well-orderable set of urelements in $\Mtail$ disjoint from $A$. Then there are some $m \geq k > n$ such that $|B \cap A_{k, m}| > \omega_n$: otherwise $B$ would be locally small and hence well-orderable in $\Mtail$. So it follows that $\Mtail \models \aleph(B) > \omega_n$. On the other hand, for every $v \in \Mtail$, let $n < \omega$ be such that $\ker(v) \subseteq \bigcup_{i \leq n}T_n$. Set $B = A_{n+1, n+1, 0}$, which witnesses $\Mtail \models \neg \exists A \in v \ \varphi (n+1, A)$.
\end{proof}

We next show that RP$^\sim$ holds in $\Mtail$ by using a supertransitive class to simulate $\Mtail$ inside any given $V(\bigcup_{i\leq n}T_i)$.

\begin{definition}
For every $m, n, n', k, k' < \omega$, define $\sigma^m_{n, n', k, k'}$ to be the bijective map $a_{n, m, k, \alpha} \mapsto a_{n', m, k', \alpha}$ from  $ A_{n, m, k}$ to $ A_{n', m, k'}$. A permutation $\sigma$ of $\A$ is \textit{row-swapping} if for every $n, m, k < \omega$, $\sigma \restriction A_{n, m, k}$ is $\sigma^m_{n, n', k, k'}$ for some $n', k' < \omega$.
\end{definition}

\begin{lemma}\label{lemma:rowswapping}
Let $\sigma$ be a row-swapping permutation that pointwise fixes $\bigcup_{k<\omega} A_{k,0}$. Then for every $x$, $x$ is $\F$-symmetric if and only if $\sigma x$ is $\F$-symmetric.
\end{lemma}
\begin{proof}
Suppose that $x$ is $\F$-symmetric and let $E$ be a locally small set that supports $x$. Then for each $A_{i, j}$, where $j > 0$, we have 
$$\sigma E \cap A_{i, j} = \sigma (E \cap \sigma^{-1} A_{i, j}) = \sigma (\bigcup_{ k < \omega} E \cap \sigma^{-1} A_{i, j, k}).$$
Since each $\sigma^{-1} A_{i, j, k}$ is some $A_{i', j, k'}$, $|E \cap \sigma^{-1} A_{i, j, k}| < \omega_j$. Thus, $|\sigma E \cap A_{i, j}| < \omega_j$, making $\sigma E$ locally small.

It remains to show $\fix_\G(\sigma E) \subseteq \sym_\G(\sigma x)$. Let $\pi \in \fix_G(\sigma E)$. The action of $\pi$ on every row $A_{i, j, k}$ follows the same permutation $\pi_j$ of $\omega_j$, so for any urelement $a_{i, j, k, \alpha}$, $\sigma \pi \sigma^{-1} (a_{i, j, k, \alpha}) = a_{i, j, k, \pi_j(\alpha)} = \pi a_{i, j, k, \alpha}$. This shows that $\pi = \sigma^{-1} \pi \sigma = \sigma \pi \sigma^{-1}$. It follows that $\pi \in \fix_\G(E)$ and hence $\pi \in \sym_\G(x)$. Thus, $\pi \sigma x = \sigma \pi \sigma^{-1} \sigma x = \sigma x.$ Therefore, $\fix_\G(\sigma E) \subseteq \sym_\G(\sigma x)$. The inverse of any row-swapping permutation is also row-swapping, so the lemma is proved.\end{proof}

\begin{lemma}\label{lemma:simulation}
For every $0 < n < \omega$, there is a definable supertransitive class $M^*$ of $V(\bigcup_{i \leq n} T_i)^{\Mtail}$ and an isomorphism $\sigma$ from $\Mtail$ to $M^*$ such that $\sigma$ pointwise fixes $\bigcup_{i < n} T_i$.
\end{lemma}
\begin{proof}
Fix $0< n < \omega$ and we work in $V(\bigcup_{i \leq n} T_i)^{\Mtail}$. For each $i < \omega$, let $P_i = \{p_j \mid j \leq i\}$ be an $(i+1)$-partition of $\omega$ such that each $p_j = \{k_j \mid k < \omega\}$ is infinite. For each $A_{n, n+i}$, since $\<A_{n, n+i, k} \mid k < \omega >$ is in $\Mtail$, $P_i$ induces an $(i+1)$-partition $\{A^j_{n,n+i} \mid j \leq i \}$ of $A_{n, n+i}$, where $A^j_{n,n+i} = \{A_{n, n+i, k_j } \mid k_j \in p_j\}$. In $V(\bigcup_{i \leq n} T_i)^{\Mtail}$, for each $k < \omega$, define

\begin{equation*}
T^*_k = 
\begin{cases*}
T_k   & $k < n$ \\
 \bigcup_{i \geq j}A^{j}_{n, n+i} & $k = n + j$
\end{cases*}
\end{equation*}
Define
$$M^* = \{ x \in V(\bigcup_{i \leq n} T_i)^{\Mtail} \mid  \exists j < \omega (\ker(x) \subseteq \bigcup_{k \leq j} T^*_k)\}.$$
\noindent $M^*$ is a supertransitive class in $\Mtail$.

Now we define the isomorphism $\sigma$ in $U$. Fix $i < \omega$.
\begin{itemize}
\item [] (i) If $j \leq i$, let $\sigma_{i, j} = \bigcup_{k <\omega} \sigma^{n+i}_{n+j, n, k, k_j}$, which injects $A_{n+j, n+i}$ into $A^j_{n, n+i}$;

\item [] (ii) If $j > i$, let $\sigma_{i, j} = \bigcup_{k <\omega} \sigma^{n+i}_{n+j, n+(j -i), k, k}$.
\end{itemize}

\noindent Let $\sigma^*$ be the canonical extension of $\bigcup_{i, j<\omega} \sigma_{i,j}$ that pointwise fixes every other urelement. Then $\sigma^*$ is a row-swapping permutation of $\A$ that pointwise fixes both $\bigcup_{k<\omega} A_{k,0} $ and $ \bigcup_{i < n} T_i$. Let $\sigma = \sigma^* \restriction \Mtail$. We have $\sigma T_k = T^*_k$ for every $k < \omega$.

For every $x \in \Mtail$, since $\ker(x) \subseteq \bigcup_{k \leq j}T_k$ for some $j < \omega$, $\ker(\sigma x) \subseteq \bigcup_{k \leq j} \sigma T_k = \bigcup_{k \leq j} T^*_k$. Then it follows from Lemma \ref{lemma:rowswapping} that $\sigma x \in \Mtail$ and hence $\sigma x \in M^*$. On the other hand, if $x \in M^*$, $\ker(x) \subseteq \bigcup_{k \leq j} T^*_k$ for some $j$ so $\ker(\sigma^{-1} x) \subseteq \bigcup_{k \leq j} T_k$; and since $\sigma^{-1}$ is row-swapping and $M^* \subseteq \Mtail$,  $\sigma^{-1} x \in \Mtail$ by Lemma \ref{lemma:rowswapping} again. Therefore, $\sigma$ is an isomorphism from $\Mtail$ to $M^*$. \end{proof}

\begin{lemma}
$\Mtail \models$ RP$^\sim$.
\end{lemma}
\begin{proof}
Suppose that $\Mtail \models \varphi(x)$ for some $x \in \Mtail$. Let $n > 0$ be such that $\ker(x) \subseteq \bigcup_{i < n}T_i$. By Lemma \ref{lemma:simulation}, there is a definable supertransitive class $M^*$ in $V(\bigcup_{i \leq n} T_i)^{\Mtail}$ and an isomorphism $\sigma$ from $\Mtail$ to $M^*$ such that $\sigma$ pointwise fixes $\bigcup_{i < n}T_i$. Thus, $M^* \models \varphi(x)$. By RP and hence RP$^\sim$ in $V(\bigcup_{i \leq n} T_i)^{\Mtail}$, there is a supertransitive set $t \in V(\bigcup_{i \leq n} T_i)^{\Mtail}$ such that $x \in t$ and $M^* \cap t \models \varphi (x)$. So $M^* \cap t$ is a desired supertransitive set in $\Mtail$ containing $x$.\end{proof}

\begin{lemma}
$\Mtail \models $ Tail $\land$ SVC.
\end{lemma}
\begin{proof}
In $\Mtail$, for every $ n, m > 0$, $\bigcup_{n \leq i  < n + m} T_{i}$ injects into $T_n$. This is because $\sigma$ defined in Lemma \ref{lemma:simulation} restricted to  $\bigcup_{n \leq i  < n + m} T_{i}$ is such an injection, since $\sym_\G(\sigma^m_{n, n', k, k'}) = \G$ for every $m, n, n', k, k' < \omega$ and hence every collection of them is $\F$-symmetric. It follows that $T_n$ is a tail of $\bigcup_{i < n} T_i$ for every $n > 0$, which implies that $\Mtail \models$ Tail. So it follows from Theorem \ref{thm:Tail+Collection->RP} and \ref{thm:FThmofSKPM} that $\Mtail \models$ SVC.\end{proof}

\section*{Funding}
The second author was supported by NSFC No. 12401001 and the Fundamental Research Funds for the Central Universities, Peking University.

\printbibliography

@incollection{levy1969definability,
  title={The definability of cardinal numbers},
  author={L{\'e}vy, Azriel},
  booktitle={Foundations of Mathematics},
  pages={15--38},
  year={1969},
  publisher={Springer}
}

@incollection{levy1966principles,
  title={On the principles of reflection in axiomatic set theory},
  author={L{\'e}vy, Azriel},
  booktitle={Studies in Logic and the Foundations of Mathematics},
  volume={44},
  pages={87--93},
  year={1966},
  publisher={Elsevier}
}

@article{gauntt1967undefinability,
  title={Undefinability of cardinality},
  author={Gauntt, Robert J},
  journal={Lectures notes prepared in connection with the Summer Institute on Axiomatic Set Theory held at University of California, Los Angeles, IV-M},
  year={1967}
}

@article{blassInjec1979,
 ISSN = {00029947},
 URL = {http://www.jstor.org/stable/1998165},
 author = {Andreas Blass},
 journal = {Transactions of the American Mathematical Society},
 pages = {31--59},
 publisher = {American Mathematical Society},
 title = {Injectivity, Projectivity, and the Axiom of Choice},
 urldate = {2025-08-18},
 volume = {255},
 year = {1979}
}

@article{YaoAxiomandForcing,
	author = {Bokai Yao},
	doi = {10.1017/jsl.2024.58},
	journal = {Journal of Symbolic Logic},
	title = {Axiomatization and Forcing in Set Theory with Urelements},
	year = {forthcoming}
}

@article{FinitenessSVC,
	author = {Horst Herrlich and Paul Howard and Eleftherios Tachtsis},
	doi = {10.1215/00294527-3490101},
	journal = {Notre Dame Journal of Formal Logic},
	number = {3},
	pages = {375--388},
	title = {Finiteness Classes and Small Violations of Choice},
	volume = {57},
	year = {2016}
}

@MISC {GlazerMOFUnboundedCards,
    TITLE = {Is every set of cardinals bounded?},
    AUTHOR = {Elliot Glazer (https://mathoverflow.net/users/109573/elliot-glazer)},
    HOWPUBLISHED = {MathOverflow},
    NOTE = {URL:https://mathoverflow.net/q/501823 (version: 2025-10-20)},
    EPRINT = {https://mathoverflow.net/q/501823},
    URL = {https://mathoverflow.net/q/501823}
}

@misc{Yao2025Plenitude,
      title={Plenitudinous Urelements and the Definability of Cardinality}, 
      author={Bokai Yao},
      year={2025},
      eprint={2508.20641},
      archivePrefix={arXiv},
      primaryClass={math.LO},
      url={https://arxiv.org/abs/2508.20641}, 
}

@article{levy1960axiom,
  title={Axiom schemata of strong infinity in axiomatic set theory},
  author={L{\'e}vy, Azriel and others},
  journal={Pacific journal of mathematics},
  volume={10},
  number={1},
  pages={223--238},
  year={1960},
  publisher={Mathematical Sciences Publishers}
}

@incollection{Montague1961-MONFAT,
	author = {Richard Montague},
	booktitle = {Essays on the Foundations of Mathematics},
	editor = {Bar{-}Hillel and Yehoshua and [From Old Catalog]},
	pages = {662--662},
	publisher = {Magnes Press},
	title = {Fraenkel's Addition to the Axioms of Zermelo},
	year = {1961}
}

\end{document}